\documentclass[graybox]{svmult}

\usepackage{type1cm}        % activate if the above 3 fonts are
\usepackage[hang,small,bf]{caption}
\usepackage{float}
\usepackage{multicol}        % used for the two-column index
\usepackage[bottom]{footmisc}% places footnotes at page bottom

\usepackage{newtxtext}       % 
\usepackage[varvw]{newtxmath}       % selects Times Roman as basic font

\usepackage{color}

\usepackage{amsmath}
\usepackage{enumerate}
\usepackage{ulem}
\usepackage{physics}

\makeindex             % used for the subject index
\allowdisplaybreaks[1]

\begin{document}
\title*{Estimating the finite-time ruin probability of a surplus with a long memory via Malliavin calculus}
\titlerunning{\sc Estimating ruin probability via Malliavin calculus}
\author{Shota Nakamura and Yasutaka Shimizu}
\institute{Shota Nakamura \at Graduate School of Fundamental Science and Engineering, Waseda University
3-4-1, Shinjuku, Okubo, Tokyo, 169-8555, Japan, \email{nakamurashota@akane.waseda.jp}
\and Yasutaka Shimizu \at Department of Applied Mathematics, Waseda University
3-4-1, Shinjuku, Okubo, Tokyo, 169-8555, Japan,  \email{shimizu@waseda.jp}}
\maketitle

\begin{abstract}
\ We consider a surplus process of drifted fractional Brownian motion with the Hurst index $H>1/2$, which appears as a functional limit of drifted compound Poisson risk models with correlated claims. This is a kind of representation of a surplus with a long memory. 
Our interest is to construct confidence intervals of the ruin probability of the surplus when the volatility parameter is unknown. 
We will obtain the derivative of the ruin probability w.r.t. the volatility parameter via  Malliavin calculus, and apply the delta method to identify the asymptotic distribution of an estimated ruin probability.
\begin{flushleft}
\vspace{1mm}
\end{flushleft}
\end{abstract} 

\section{Introduction}
In the classical ruin theory initiated by Lundberg \cite{Lundberg}, the insurance surplus is described by a drifted compound Poisson process such as 
\begin{align}
\label{compound}
X_t=x + ct - \sum_{i=1}^{N_t} U_i
\end{align}
where $x,c>0$, $N$ is a Poisson process, and $U_i$'s are IID random variables with mean $\mu$, representing claim sizes. 
 One of the direction to extend the model is the following drifted L\'{e}vy  surplus $(R_t)_{0<t<T}$:
\begin{align}
\label{levy}
R_t=u+dt+\sigma W_t -V_t, 
\end{align}
where $u,d,\sigma >0$, $W$ is a Brownian motion and $V$ is a L\'{e}vy subordinator. The model (\ref{levy}) is a natural extension of (\ref{compound}) and considers claim sizes with stationary independent increment. Statistical inference for ruin probability based on the model (\ref{levy}) has been studied by many authors; see, e.g., Asmussen and Albrecher \cite{aa10}, Shimizu \cite{Shimizu1} and the references therein. However, such an independent assumption is often unrealistic in a certain insurance contract because large claims can be successive once a large claim has occurred. Therefore, it would be better to assume that $ (U_i)_{i\in \mathbb{N}}$ are correlated.

Michna \cite{Michna} assumes that there exists a constant $D\in (0,1)$ and a slowly varying function $L$: 
$L(tx)\sim L(x)$ as $x\to \infty$ for any $t>0$, such that 
\begin{align}
\mathbb{E}[U_iU_{i+k}] \sim k^{-D}L(k),\quad k\to \infty, 
\end{align}
under which the process $X$ has a {\it long memory}: $\sum_{k=1}^\infty Cov(U_1,U_{k}) = \infty$. 
Considering a sequence of such a long memory surplus processes $X^n=\left( X_t^n\right)_{t \in [0,T]}$ indexed by $n=1,2,\dots,$:
\begin{align*}
X_t^n=x_n+c_n t -\sum_{i=0}^{N_t^n} U_i, 
\end{align*}
where $x_n, c_n$ are positive sequences, $\left(N_t^n\right)_{t \in [0,T]}$ is a Poisson process with the intensity $n$ and  $U_i$'s are correlated random variables as above. Then, according to Theorem 3 in Michna \cite{Michna}, there exists a norming sequence $(\eta_n)_{n\in \mathbb{N}}$ and some constants $u$ and $\theta$ such that 
the process $X^n/\eta_n$ converges weakly in a functional space $D[0, \infty)$, a space of c\`{a}dl\`{a}g functions with the Skorokhood topology: 
\begin{align*}
\frac{X_t^n}{\eta_n}\xrightarrow{d} u+ \theta t -W_t^H \quad \text{in} \quad \mathcal{D}[0, \infty) \quad(n\rightarrow \infty),
\end{align*}
where $W^H$ is a fractional Brownian motion with the Hurst parameter $H\in (\frac{1}{2},1)$. In such a way, the surplus driven by a fractional Brownian motion naturally appears as a limit of a Poissonian model with a long memory. 

 Some earlier works model a surplus by fractional Brownian motions. Ji and Robert \cite{Ji} model the surplus of insurance and reinsurance companies as the two-dimensional fractional Brownian motion and derives asymptotic of the ruin probability when the initial capital tends to infinity. 
Cai and Xiao \cite{Cai} consider a drifted mixed fractional Brownian motion as a surplus model and estimate the ruin probability with an unknown drift parameter. %However, to the author's knowledge, estimation for the ruin probability of a surplus modeled by the drifted fractional Brownian motion with an unknown volatility parameter is still unclear. In the following, 
In this paper, we are interested in the following drifted fractional Brownian motion as a surplus model: 
\begin{align}
\label{model}
X_t=u+\sigma \theta t-\sigma W_t^H,
\end{align}
where $\theta>0$ and $H \in (\frac{1}{2},1)$ are known parameters and $\sigma >0$ is an unknown parameter.
Since our model is a normalized limit of a classical type surplus with a known premium rate $c_n$, the drift will be known under a suitable scaling. 
Therefore we assume $\theta$ is known although the scaling parameter $\sigma$ is unknown.

Our interest is to estimate the finite-time ruin probability: for any $T \in (0,\infty)$,
\begin{align*}
\Psi_{\sigma}(u,T):=\mathbb{P}\left( \underset{0\leq t\leq T}{\inf} X_t <0 \right).
\end{align*}
from the past surplus data. 

The paper is organized as follows: In Section 2, we prepare some notation and give a brief review of Malliavin calculus. In Section 3, we provide a result on estimating the volatility parameter $\sigma$ and the ruin probability by the delta method.
 In this procedure, the partial derivative $\frac{\partial}{\partial_{\sigma}} \Psi_{\sigma}(u,T)$ is required to obtain confidence intervals of the ruin probability, so we derive its explicit form using the integration by parts formula in Malliavin calculus in Section 4.

\section{Preliminaries}
\subsection{Notation}
We use the following notations.
\begin{itemize}
\item $A\underset{\sim}{<}B$ means that there exists a universal constant $c>0$ such that $A\leq cB$.

\item The partial derivative of the function $f$ at the point $x\in \mathbb{R}^d$ with respect to the $i$-th variable is denoted by $\partial_i f(x)$.

\item Let $H_n(\cdot)$ denote the $n$-th order \textit{Hermite polynomial}, which is defined by
\begin{align*}
H_n(x)=\frac{(-1)^n}{n!}e^{\frac{x^2}{2}}\frac{d^n}{dx^n}\left(e^{-\frac{x^2}{2}} \right), \quad (n\geq1).
\end{align*}  
\item $\mathcal{D}(A)$: the Skorokhod space on the set $A \subset \mathbb{R}_{+}$. 

\item Let $C_\uparrow^{\infty}(\mathbb{R}^n)$ be the set of all infinitely continuously differentiable functions $f : \mathbb{R}^n \rightarrow \mathbb{R}$ such that $f$ and all of its partial derivatives are of polynomial growth. %（$C_p^\infty$の$p$が定数や$L^p$などと紛らわしいので，$C_\uparrow^\infty$などを使うことが多い）.
\item Denote by $C_b^{\infty}(\mathbb{R}^n)$ the set of all infinitely continuously differentiable functions $f: \mathbb{R}^n \rightarrow \mathbb{R}$ such that  $f$ and all of its partial derivatives are bounded.
\item For any $p>1$ and $f, g \in L^p(\mathbb{R})$ we define the convolution $f*g$ as
\begin{align*}
f*g(x):=\int_{\mathbb{R}} f(x-y)g(y) dy.
\end{align*}
\item We denote the gamma function $\Gamma(\cdot)$ and the beta function $B(\cdot,\cdot)$ by 
\begin{align*}
\Gamma(x)&=\int_0^{\infty} t^{x-1}e^{t}dt, \quad (x>0)\\
B(x,y)&=\int_0^1 t^{x-1} \left(1-t \right)^{y-1} dt, \quad (x,y>0).
\end{align*}
\item We denote the left and right-sided fractional integrals $I_{a\pm}^{\alpha}f(\cdot)$ and derivatives $D_{a\pm}^{\alpha}f(\cdot)$ by
\begin{align*}
I_{a+}^{\alpha}f(x)&=\frac{1}{\Gamma(\alpha)}\int_a^x (x-y)^{\alpha-1}f(y)dy,\\
I_{b-}^{\alpha}f(x)&=\frac{1}{\Gamma(\alpha)}\int_x^b(y-x)^{\alpha-1}f(y)dy,\\
D_{a+}^{\alpha}g(x)&=\frac{1}{\Gamma(1-\alpha)} \left(\frac{g(x)}{(x-a)^{\alpha}}+\alpha \int_a^x \frac{g(x)-g(y)}{(x-y)^{\alpha+1}} dy \right),\\
D_{b-}^{\alpha} g(x)&=\frac{1}{\Gamma(1-\alpha)}\left( \frac{g(x)}{\left(b-x\right)^{\alpha}}+\alpha \int_x^b \frac{g(x)-g(y)}{(y-x)^{\alpha+1}} dy \right),
\end{align*}
for any $0<\alpha<1$, $x \in (a,b)$, $f\in L^1(a,b)$ and $g \in I_{a+}^{\alpha}(L^p)$ $\left(\text{resp.}\  g \in  I_{b-}^{\alpha}(L^p)\right)$ where $p>1$ (see Samko \textit{et al.} \cite{Samko} for details).
\end{itemize}

\subsection{Malliavin calculus}
This section briefly introduces the Maliavin calculus based on Chapters 1 and 5 in  Nualart \cite{Nualart2}.   In the sequel, we denote by $G$  a real separable Hilbert space.

\subsubsection{Malliavin calculus on a real separable Hilbert space}

\begin{definition}
We say that a stochastic process $(W_g)_{g\in G}$ is an isonormal Gaussian process associated with the real separable Hilbert space $G$ if $W$ is a centered Gaussian family of random variables such that $\mathbb{E}[W_hW_g]=\left\langle h, g \right\rangle_G$ for any $h, g \in G$.
\end{definition}
In the following, we assume that the isonormal gaussian $W(\cdot)$ is defined on the complete probability space $(\Omega, \mathcal{G}, \mathbb{P})$ where $\mathcal{G}$ is the $\sigma$-algebra generated by $W$ in this paper.
\begin{definition}
Let $\sigma$-algebra $\mathcal{G}$ be the $\sigma$-algebra generated by an isonormal Gaussian $W$. If a random variable $F : \Omega \rightarrow \mathbb{R}$ satisfies 
\begin{align}
\label{smooth}
F=f\left(W(g_1), \ldots, W(g_n) \right) \quad(f \in C_{\uparrow}^{\infty} (\mathbb{R}), g_1, \ldots, g_n \in G),
\end{align}
$F$ is called a smooth random variable, and the set of all such random variables is denoted by $\mathcal{S}_G$.
\end{definition}

\begin{definition}
The derivative of a smooth random variable $F$ of the form $(\ref{smooth})$  is the $G$-valued random variable given by
\begin{align*}
D^{G} F=\sum_{i=1}^n \partial_i f \left(W(g_1), \ldots , W(g_n) \right) g_i.
\end{align*}
\end{definition}

Since the operator $D^G$ defined in Definition 3 is a closable operator, we can extend $D^G$ as a closed operator on  $\mathbb{D}_G^{1,p} := \overline{\mathcal{S}_G}^{\|\cdot \|_{1,p}}$ where the seminorm $\| \cdot \|_{1,p}$ on $\mathcal{S}_G$ is defined by
\begin{align*}
\|F\|_{1,p} := \left(\mathbb{E}[F^p] +\mathbb{E}\left[ \|D^GF\|_G^p \right] \right)^{\frac{1}{p}}
\end{align*}
for any $p \geq 1$.

The above definitions can be exended to Hilbert-valued random variables. Consider the family $\mathcal{S}_G(V)$ of $V$-valued smooth random variables of the form
\begin{align*}
F=\sum_{i=1}^n F_i v_i \quad(v_i \in V, F_i \in \mathcal{S}_G),
\end{align*}
Define $D^GF:=\sum_{i=1}^n D^GF_i \otimes v_j$. Then $D^G$ is a closable operator from $\mathcal{S}_G(V)$ into $L^p(\Omega; G\otimes V)$ for any $p \geq 1$. Therefore, $D^G$ is a closed operator on $\mathbb{D}_G^{1,p}(V) = \overline{\mathcal{S}_G(V)}^{\| \cdot \|_{1,p,V}}$ for the seminorm $\| \cdot \|_{1,p}$ determined by 
\begin{align*}
\|F\|_{1,p,V}:= \left( \mathbb{E}\left[ \|F\|_V^p \right] + \mathbb{E}\left[ \|D^GF\|_{G\otimes V}^p \right] \right)^{\frac{1}{p}},
\end{align*}
 on $\mathcal{S}_G(V)$. In particular, we define $\mathbb{D}_G^{1,\infty}$ and $\mathbb{D_G}^{1,\infty}(V)$ by 
 \begin{align*}
\mathbb{D}_G^{1,\infty}:= \bigcap_{p=1}^{\infty} \mathbb{D}_G^{1,p}, \quad \mathbb{D_G}^{1,\infty}(V) :=\bigcap_{p=1}^{\infty}  \mathbb{D_G}^{1,p}(V).
 \end{align*}
 The following proposition is the chain rule for $D^G$.
 \begin{proposition}
 Suppose that $F=(F^1, \ldots , F^m)$ is a random vector whose componets belong to $\mathbb{D_G}^{1,\infty}$. Let $f \in C_p^{\infty}(\mathbb{R}^m)$. Then $f(F) \in \mathbb{D_G}^{1, \infty}$, and we have
 \begin{align*}
 D^G(f(F))=\sum_{i=1}^m \partial_i f(F) D^G F^i.
 \end{align*}
 \end{proposition}
 Next, we consider the divergence operator.
\begin{definition}
 The divergence operator $\delta^G$ is an unbounded operator on $L^2(\Omega ; G)$ with values in $L^2(\Omega)$ such that:
 \begin{itemize}
 \item[(1)] The domain of $\delta^G$ , denoted by Dom$\delta$, is the set of stochastic processes $u_{\cdot} \in L^2(\Omega ; G)$ such that
 \begin{align*}
 \left| \mathbb{E} \left[ \left\langle D^GF, u_{\cdot} \right\rangle_G \right] \right| \leq c(u) \| F\|_{L^2(\Omega)},
 \end{align*}
 for any $F \in \mathbb{D}_G^{1,2}$, where $c(u)$ is some constant depending on $u$.
 \item[(2)]If $u_{\cdot}$ belongs to Dom$\delta^G$, then $\delta^G(u)$ is characterized by
 \begin{align*}
 \mathbb{E}\left[ F\delta^G(u)\right] = \mathbb{E} \left[ \left\langle D^G F, u \right\rangle \right],
 \end{align*}
 for any $F\in \mathbb{D}_G^{1,2}$.
 \end{itemize}
\end{definition}
The following proposition allows us to factor out a scalar random variable in a divergence.
\begin{proposition}
Let $F \in \mathbb{D}_G^{1,2}$ and $u \in \text{Dom}\ \delta^G$ such that $Fu \in L^2(\Omega ; G)$. Then $Fu \in \text{Dom}\ \delta^G$ and it follows that
\begin{align*}
\delta^G(Fu)=F\delta^G(u)-\left\langle D^GF, u\right\rangle_G.
\end{align*}
\end{proposition}

\subsubsection{Malliavin calculus for the fractional Brownian motion}
In this subsection, we introduce the fractional Brownian motion and the Hilbert space associated with the fractional Brownian motion.
\begin{definition}
A cented Gaussian process $(W_t^H)_{t\geq0}$ is called fractional Brownian motion of Hurst index $H \in (0,1)$ if it has the covariance function
\begin{align*}
R_H(t,s) := \mathbb{E} [W_t^H W_s^H]=\frac{1}{2}(s^{2H}+t^{2H}-|t-s|^{2H}).
\end{align*}
\end{definition}
In this paper we will only use the fractional Brownian motions with Hurst index $H>\frac{1}{2}$. We denote by $\mathcal{E}$ the set of step functions on [0, T]. Let $\mathcal{H}$ be the Hilbert space defined as the closure of $\mathcal{E}$  with respect to the scalar product
\begin{align*}
\left\langle 1_{[0,,t]}, 1_{[0,s]} \right\rangle_{\mathcal{H}} =R_H(t,s),
\end{align*}
which yields that, for any $\phi, \psi \in \mathcal{H}$,
\begin{align*}
\left\langle \psi , \phi \right\rangle_{\mathcal{H}} := H(2H-1)\int_0^T \int_0^T |r-u|^{2H-2} \psi(r) \phi(u) dr du.
\end{align*}
It is easy to see that the covariance of fractional Brownian motion can be written as
\begin{align*}
\left\langle 1_{[0,t]} , 1_{[0,s]} \right\rangle_{\mathcal{H}} &=H(2H-1)\int_0^t \int_0^s |r-u|^{2H-2}dudr \\
&=R_H(t,s).
\end{align*}
Therefore, fractional Brownian motion $W^H$ can be expressed as $W_t^H=W(1_{[0,t]})$ for the isonormal Gaussian $W$ associated with the Hilbert space $\mathcal{H}$.
Consider the square integrable kernel 
\begin{align*}
K_H(t,s) := c_H s^{\frac{1}{2}-H} \int_s^t (u-s)^{H-\frac{3}{2}}u^{H-\frac{1}{2}}du,
\end{align*}
where $c_H=\left[ \frac{H(2H-1)}{B(2-2H, H-\frac{1}{2})}\right]^{\frac{1}{2}}$ and $t>s$.
Define the isometric function $K_H^* : \mathcal{E} \rightarrow  L^2(0,T)$ by
\begin{align*}
(K_H^*\phi)(s):=\int_s^T \phi(t)\frac{\partial K_H}{\partial t}(t,s) dt,
\end{align*}
then the operator $K_H^*$ is an isometry between $\mathcal{E}$ and $L^2(0,T)$ that can be extended to the Hilbert space $\mathcal{H}$. The operator $K_H^*$ can be expressed in terms of fractional integrals:
\begin{align*}
(K_H^*\phi)(s)=c_H\Gamma(H-\frac{1}{2})s^{\frac{1}{2}-H}(I_{T-}^{H-\frac{1}{2}}\cdot^{H-\frac{1}{2}}\phi(\cdot))(s).
\end{align*}
Finally, we consider the Malliavin calculus on $\mathcal{H}$. For the sake of simplicity, we will use the notation $D^{W^H}, \mathbb{D}_{W^H}^{1,p}$, and $\delta^{W^H}$ as the derivative operator, the domain of the derivative, and the divergence operator associated with the Hilbert space $\mathcal{H}$, respectively. In the sequel, we present two results on the derivative operator for fractional Brownian motion.
\begin{proposition}
For any $F\in \mathbb{D}_{W_H}^{1,2}=\mathbb{D}_{L^2(0,T)}^{1,2}$
\begin{align*}
K_H^*D^{W_H}F=D^{L^2([0,T])}F.
\end{align*}
\end{proposition}
\begin{proposition}
$\underset{0 \leq t \leq T}{\sup}(W_t^H-\theta t) $ belongs to $\mathbb{D}_{W_H}^{1,2}$ and it holds $D_t^{W_H} \underset{0 \leq t \leq T}{\sup}(W_t^H-\theta t) =1_{[0,\tau]}(t)$, for any $t \in [0,T]$, where $\tau$ is the point where the supremum is attained.
\end{proposition}
\begin{proof}
For the proof, see Lemma 3.2 in Florit and Nualart \cite{Florit}.
\end{proof}

\section{Statistical problems}
Suppose that we have the past surplus data in $[0,T_0]$-interval at discrete time points $\frac{[kT_0]}{n}$ ($k=0, 1,  \cdots, n$). Our goal is to estimate the finite-time ruin probability for each $T \in (0,\infty]$ from the discrete data $(X_{\frac{[kT_0]}{n}})_{k\in \{0, 1, \cdots, n\}}$.

\subsection{Estimation of $\sigma$}
Fix $0<H<\frac{3}{4}$. We use the results of Corcuera and Nualart \cite{Nualart1}  to construct the estimator for the true value of $\sigma$ by using a power variations of the order $p>0$. We define the power variation of $(X_t)$ as follows:
\begin{align*}
V_p^n(X)_t:=\sum_{i=1}^{[nt]} \left|X_{\frac{i}{n}}-X_{\frac{i-1}{n}}\right|^p.
\end{align*}
Defining  the estimator $\hat{\sigma}_{k, t ,n}$ of $\sigma_0$ as
\begin{align}
\label{sigmahat}
\widehat{\sigma}_{k, t ,n}:=\left( \frac{V_p^n(X)_t}{c_pn^{1-pH}t}\right)^{\frac{1}{p}},
\end{align}
where $c_p=\frac{2^{\frac{p}{2}}\Gamma\left(\frac{p+1}{2}\right) }{\Gamma\left(\frac{1}{2}\right)}$. We obtain the asymptotic normality of $\widehat{\sigma}_{k, t ,n}$ as in the following theorem.
\begin{theorem}
Let $p>1$ and $0<H<\frac{3}{4}$. Then
\begin{align*}
\sqrt{n} \left( \left(\widehat{\sigma}_{k,t,n}\right)^p - \sigma^p \right) \xrightarrow{\mathcal{L}} \frac{v_1 \sigma^p}{c_{p}} W_t  \quad (n \rightarrow \infty),
\end{align*}
in law in the space $\mathcal{D}([0,T_0])$ equipped with the Skorohod topology, where
\begin{align*}
 v_1^2&=\mu_p+2\sum_{j\geq1} \left(\gamma_p\left(\rho_H(j)\right)- \gamma_p(0)\right),\\
\mu_p&=2^p\left(\frac{1}{\sqrt{\pi}}\Gamma\left(p+\frac{1}{2}\right)-\frac{1}{\pi}\Gamma\left(\frac{p+1}{2}\right)^2\right),\\
\gamma_p(x)&=(1-x^2)^{\frac{p+1}{2}}2^p\sum_{k=0}^{\infty} \frac{(2x)^{2k}}{\pi (2k)!}\Gamma\left(\frac{p+1}{2}+k\right)^2,\\
\rho_H(n)&=\frac{1}{2}\left((n+1)^{2H}+(n-1)^{2H}-2n^{2H}\right),
\end{align*}
and  $(W_t)_{t\in [0,T_0]}$ is a Brownian
motion independent of the fractional Brownian motion $W_t^H$.
\end{theorem}
\begin{proof}
It it sufficient to prove that
\begin{align}
\label{induction}
n^{-\frac{1}{2}+pH}V_{k,p}^n(u+\sigma\theta t) \rightarrow 0 \quad(n\rightarrow \infty),
\end{align}
in probability, uniformly on $[0,T_0]$, by Corcuera and Nualart \cite{Nualart1}, p.727. Thus, we show that 
\begin{align}
\label{induction2}
n^{-\frac{1}{2}+pH}V_{k,p}^n(u+\sigma \theta t)=n^{-\frac{1}{2}+pH} \sum_{i=1}^{[nt]-k+1} \left| \Delta_k(u+\sigma \theta \frac{i-1}{n}) \right|^p \rightarrow 0 \quad \text{a.s.},
\end{align}
as $n\rightarrow \infty$, 
by induction on k. Note that $(1-H)p>\frac{1}{2}$, when $k=1$ we have
\begin{align*}
n^{-\frac{1}{2}+pH} \sum_{i=1}^{[nt]} \left| \Delta_1(u+\sigma \theta \frac{i-1}{n})\right|^p&=n^{-\frac{1}{2}+pH}\sum_{i=1}^{[nt]} \left|\sigma \theta(\frac{i}{n}-\frac{i-1}{n}) \right|^p\\
&=n^{-\frac{1}{2}+pH}|\sigma \theta|^p\frac{[nt]}{n^p}\\
&\rightarrow 0,
\end{align*}
as $n \rightarrow \infty$.
In the same way , we have
\begin{align*}
n^{-\frac{1}{2}+pH} \sum_{i=2}^{[nt]+1} \left| \Delta_1(u+\sigma \theta \frac{i-1}{n})\right|^p \rightarrow 0\quad (n\rightarrow \infty).
\end{align*}
Assuming  
\begin{align*}
&n^{-\frac{1}{2}+pH} \sum_{i=1}^{[nt]-k+2} \left| \Delta_{k-1}(u+\sigma \theta \frac{i-1}{n})\right|^p \rightarrow 0 \quad(n \rightarrow \infty),\\
&n^{-\frac{1}{2}+pH} \sum_{i=1}^{[nt]-k+2} \left| \Delta_{k-1}(u+\sigma \theta \frac{i}{n})\right|^p \rightarrow 0 \quad(n \rightarrow \infty),
\end{align*} 
 holds, we get 
 \begin{align*}
 n^{-\frac{1}{2}+pH} &\sum_{i=1}^{[nt]-k+1} \left| \Delta_{k}(u+\sigma \theta \frac{i-1}{n})\right|^p \\
 &\leq n^{-\frac{1}{2}+pH}\sum_{i=1}^{[nt]-k+2} \left\{ \left| \Delta_{k-1}(u+\sigma \theta \frac{i}{n})\right|^p+ \left| \Delta_{k-1}(u+\sigma \theta \frac{i-1}{n})\right|^p \right\} \\
&\rightarrow 0 ,
 \end{align*}
 as $n \rightarrow \infty$ by $\Delta_{k}X_{i-1}=\Delta_{k-1}X_i-\Delta_{k-1}X_{i-1}$.
 Therefore, since $(\ref{induction2})$ holds for any $k \in \mathbb{N}$, we get 
 \begin{align*}
 \mathbb{P}\left( \sup_{0 \leq t \leq T} \left| n^{-\frac{1}{2}+pH}V_{k,p}^n(u+\sigma \theta t) \right| > \varepsilon \right) &\leq \frac{1}{\varepsilon^2} \left( n^{-\frac{1}{2}+pH}V_{k,p}^n(u+\sigma \theta T) \right)^2 \\
&\rightarrow 0.
 \end{align*}
 for any $\varepsilon >0$ and this proof is completed.
\end{proof}

\subsection{Simulation-based inference for $\Psi_{\sigma}(u, T)$}
Using the estimator $\widehat{\sigma}_{k,t,n}$ of $\sigma_0$  given in $(\ref{sigmahat})$, we can estimate $\Psi_{\sigma}(u,T)$ by
\begin{align*}
\widehat{\Psi}_{k,t,n}(u, T):=\Psi_{\widehat{\sigma}_{k,t,n}}(u, T),
\end{align*}
and, due to the delta method(cf. van der Vaart \cite{van der varrt} P.374), it follows that
\begin{align*}
\sqrt{n}\left( \widehat{\Psi}_{k,t,n}(u,T)-\Psi(u,T)\right)\rightarrow \partial_{\sigma} \Psi_{\sigma_0}(u,T) \frac{1}{p} \frac{v_1\sigma_0}{c_{k,p}}W_t  \quad (n\rightarrow \infty),
\end{align*}
in law in the space $\mathcal{D}([0,T_0])$ equipped with the Skorohod topology, if $\Psi_{\sigma}(u,T)$ is differentiable at the true volatility parameter $\sigma$. This leads us an $\alpha$-confidence interval for $\Psi_{\sigma}(u,T)$ such as 
\begin{align}
I_{\alpha}(\Psi):=\Bigg[ \widehat{\Psi}_{k,T_0,n}(u,T) \pm \frac{z_{\alpha/2}}{\sqrt{n}}|\partial_{\sigma}\Psi_{\widehat{\sigma}_{k,T_0,n}}(u,T)|\frac{\sqrt{T_0} v_1\widehat{\sigma}_{k,T_0,n}}{pc_{k,p}}\Bigg], %\notag \\
%&\qquad \widehat{\Psi}_{k,T_0,n}(u,T)+\frac{z_{\alpha/2}}{\sqrt{n}}\partial_{\sigma}\Psi_{\widehat{\sigma}_{k,T_0,n}}(u,T)\frac{\sqrt{T_0} v_1\widehat{\sigma}_{k,T_0,n}}{pc_{k,p}}\Big] 
\end{align}
where $[a\pm b]$ stands for the interval $[a-b,a+b]$ for $b>0$, and $z_{\alpha}$ stands for the upper $\alpha$-quantile.

 Now, the problem is to compute the following quantity;
 \begin{equation}
 \label{greeks}
 \partial_{\sigma}^k \Psi_{\sigma}(u,T)=\left(\frac{\partial}{\partial \sigma}\right)^k \mathbb{P}\left(\underset{0\leq t\leq T}{\inf}X_t^{\sigma}<0\right),\quad k=0,1.
 \end{equation}
 \begin{itemize}
 \item For $k=0$: Since $\Psi_{\sigma}(u,T)$ dose not have a closed expression, we will compute it by the Monte Carlo simulation for a given value of $\sigma$, that is, we generate sample paths of $X_t=u+\sigma\theta t- \sigma W_t^H$ for given $\sigma$, say $\left(X^{(k)}\right)_{k=1,2,\ldots,m}$ independent each other, obsearve
\begin{align*}
\widetilde{\Psi}(u,T)=\frac{1}{m}\sum_{k=1}^{m}\mathbf{1}_{\{ \tau^{(k)}\leq T\} },\quad \tau^{(k)}:=\inf\{t>0|X_t^{(k)}<0\},
\end{align*}
which goes to the true $\Psi_{\sigma}(u,T)$ almost surely as $m \rightarrow \infty$ by the strong law of large numbers.However, the event of ruin in $[0,T]$ is often very rare and the most of the indicators of summand will be zero, which will underestimate the true value $\Psi_{\sigma}(u,T)$. Changing the measure $\mathbb{P}$ into a suitable one , more efficient sampling procedure: importance sampling, will be proposed.

\item For $k=1$: Computing $\partial_{\sigma}\Psi_{\sigma}(u,T)$ is not straightforward because the integrand of the righthand side of ($\ref{greeks}$) is not differentiable in $\sigma$, and we can not differentiate it under the expectation sign $\mathbb{E}$. Moreover, computing numerically, e.g., for small $\epsilon>0$,
\begin{align}
\label{azq}
\frac{\Psi_{\sigma+\epsilon}(u,T)-\Psi_{\sigma}(u,T)}{\epsilon} \quad \text{or} \quad \frac{\Psi_{\sigma+\epsilon}(u,T)-\Psi_{\sigma-\epsilon}(u,T)}{2\epsilon},
\end{align}
we have to compute $\Psi_{\sigma+\epsilon}$ and $\Psi_{\sigma - \epsilon}(u,T)$ separeately, which usually takes much time. In addition, since the accuracy of the calculation in $(\ref{azq})$ depends on $\varepsilon$, the problem of determining the value of $\varepsilon$ also arises. Importance sampling can give the fast convergence with the variance reduction.
 \end{itemize}

 \section{Differentiability of $\Psi_{\sigma}$}
 Fix $H>\frac{1}{2}$. In this section, we discuss the differentiability of $\Psi_{\sigma}(u,T)$ with respect to $\sigma$ after the ideas of Lanjri and Nualart \cite{Nualart3} or  Gobet and Kohatsu-Higa \cite{gobet}.
 \begin{theorem}
 The finite-time ruin probability $\Psi_{\sigma}(u,T)$ is differentiable with respect to $\sigma$ and we have
\begin{align}
\label{main result}
 \partial \sigma \Psi_{\sigma}(u,T)=\mathbb{E}  \left[ 1_{[\sigma \underset{0 \leq t \leq T}{\sup} (W_t^H- \theta t) <u]} \delta^{W^H} \left(   \frac{u_A(\cdot) \underset{0 \leq t \leq T}{\sup}(W_t^H- \theta t)}{\sigma \int_0^T \psi(Y_t)dt}  \right) \right]
\end{align}
where 
\begin{align}
u_A(t)&:=\frac{d_H}{c_H \Gamma(H-\frac{1}{2})}t^{\frac{1}{2}-H}D_{T-}^{H-\frac{1}{2}} \left[(\cdot)^{2H-1}D_{0+}^{H-\frac{1}{2}}\left((\cdot)^{\frac{1}{2}-H}\psi(Y_{\cdot})\right) (\cdot)\right](t) \label{ua}\\
&=\frac{d_H}{c_HB\left(H-\frac{1}{2}, \frac{3}{2}-H\right) \Gamma \left(\frac{3}{2}-H \right)} t^{\frac{1}{2}-H} \nonumber\\
&\times \Biggl\{ \frac{1}{(T-t)^{H-\frac{1}{2}}} \left[ \psi(Y_t) + \left( H-\frac{1}{2} \right)t^{2H-1} \int_0^t \frac{t^{\frac{1}{2}-H} \psi(Y_t)- s^{\frac{1}{2}-H}\psi(Y_s)}{(t-s)^{H+\frac{1}{2}}}ds \right] \nonumber\\
&+\left(H-\frac{1}{2} \right) \int_t^T \biggl\{ \frac{\psi(Y_t)+\left(H-\frac{1}{2} \right) t^{2H-1} \int_0^t \frac{t^{\frac{1}{2}-H}\psi(Y_t)-u^{\frac{1}{2}-H}\psi(Y_u)}{(t-u)^{H+\frac{1}{2}}}du}{(s-t)^{H+\frac{1}{2}}} \nonumber\\
&-\frac{\psi(Y_s)+\left(H-\frac{1}{2} \right) s^{2H-1} \int_0^s \frac{s^{\frac{1}{2}-H}\psi(Y_s)-u^{\frac{1}{2}-H}\psi(Y_u)}{(s-u)^{H+\frac{1}{2}}}du}{(s-t)^{H+\frac{1}{2}}}\biggr\} ds\Biggr\}\nonumber\\
\label{y}
Y_t&:=8\left(4 \int_0^T \int_0^T \frac{|W_s^H- \theta s -(W_u^H- \theta u)|^r}{|s-u|^{m+2}}dsdu \right)^{\frac{1}{r}} \frac{m+2}{m} t^{\frac{m}{r}},
\end{align}

for any even integers r, m such that $rH>m+2$ and $\psi \in C_b^{\infty}(\mathbb{R}_{+})$ satisfies 
\begin{equation*}
  \psi(x)=\left\{\begin{matrix}
    1\quad(x\leq \frac{u}{2\sigma})\\
    0\quad(x \geq \frac{u}{\sigma})
  \end{matrix}\right. .
\end{equation*}
 \end{theorem}
In the proof of $(\ref{main result})$, it suffices to show that 
\begin{align}
\label{submain}
&\partial \sigma\mathbb{E} \left[ \phi (\sigma \underset{0\leq t \leq T}{\sup} (W_t^H- \theta t) )\right]\notag \\
&=\mathbb{E}  \left[ \phi \left(\sigma \underset{0 \leq t \leq T}{\sup} (W_t^H- \theta t)  \right) \delta^{W^H} \left(   \frac{u_A(\cdot) \underset{0 \leq t \leq T}{\sup}(W_t^H- \theta t)}{\sigma \int_0^T \psi(Y_t)dt}  \right) \right]
\end{align}
 holds for a sufficiently smooth function $\phi:\mathbb{R}_{+} \rightarrow \mathbb{R}$ instead of  $1_{(-\infty, 0)}(\cdot)$ in $(\ref{main result})$ from the density argument.
 In the following, we impose the following conditions on $\phi$:
\begin{itemize}
\item[(1)] $\phi \in C_b^{\infty}(\mathbb{R}_{+})$,
\item[(2)] The function $\phi$ is constant on $[0,u]$.
\end{itemize}
We prove the following properties of $Y_t$. 

\begin{lemma}
For $Y_t$ defined in $(\ref{y})$, the following (1)--(3) hold.
\begin{itemize}
\item[(1)] $|W_t^H- \theta t| \leq Y_t \quad (t\in [0,T])$.
\item[(2)] $\psi(Y_t) \in \mathbb{D}_{W^H}^{1, \infty} \quad (t\in [0,T])$. 
\item[(3)] There exists a  function $\alpha:\mathbb{R}\rightarrow \mathbb{R}_{+}$, with $\lim_{q \rightarrow \infty} \alpha(q) = \infty$ , such that, for any $q \geq 1$, one has: $\forall t \in  [0, T] \quad E[Y_t^q] \leq C_q t^{\alpha(q)}$.
\item[(4)] $\left(\int_0^T \psi(Y_t)dt\right)^{-1}\in L^p(\Omega)\quad (p\geq1)$.
\end{itemize}
\end{lemma}

\begin{proof}
We can show (1) and (2) similarly as in the proof of  Lemma 2.1 in  Gobet and Kohatsu-Higa \cite{gobet}. (3) follows immediately from (2). Next, we prove (4). It is sufficient to show   that 
\begin{align*}
\mathbb{P}\left( \int_0^T \sigma K_H(\tau, t) \Phi(Y_t) dt < \varepsilon \right) = O(\varepsilon^p) \quad(\varepsilon \downarrow 0),
\end{align*}
 holds from Nualart \cite{Nualart2}, p.133, Lemma 2.3.1. Since 
 \begin{align*}
 \int_0^T \psi(Y_t) dt &=\int_0^{\frac{\epsilon}{\sigma}} \psi(Y_t) dt + \int_{\frac{\epsilon}{\sigma}}^T \psi(Y_t) dt\\
 &\geq \frac{\epsilon}{\sigma},
 \end{align*}
 holds on $\left[  \frac{u}{2\sigma} > Y_{\frac{\epsilon}{\sigma}} \right]$, we obtain 
 \begin{align*}
 \left[ \sigma \int_0^T \psi(Y_t)dt < \epsilon \right] \subset \left[  \frac{u}{2\sigma} \leq Y_{\frac{\epsilon}{\sigma}} \right].
 \end{align*}
 Therefore, for any $q \geq 1$ such that $\alpha(q) \geq p$ we have
 \begin{align*}
 \mathbb{P} \left( \int_0^T \sigma \psi(Y_t) dt < \epsilon \right) &\leq \mathbb{P} \left( \frac{u}{2\sigma} \leq Y_{\frac{\epsilon}{\sigma}} \right)\\
 &\leq \left( \frac{u}{2\sigma} \right)^{-q} \mathbb{E}\left[ Y_{\frac{\epsilon}{\sigma}}^q \right]\\
 &\leq  \left( \frac{u}{2\sigma} \right)^{-q} \left( \frac{\epsilon}{\sigma} \right)^{\alpha(q)}\\
 &\leq O(\epsilon^q).
 \end{align*}
\end{proof}

\begin{proposition}
\begin{align*}
\left(D_s^{W_H} \phi \left( \sigma \underset{0 \leq t \leq T}{\sup} (W_t^H- \theta t) \right) \right) \psi(Y_t)=\phi'(\sigma \underset{0\leq t \leq T}{\sup} (W_t^H -\theta t)) \sigma \psi(Y_t)
\end{align*}
\end{proposition}
\begin{proof}
The proof is analogous to the proof of Gobet and Kohatsu-Higa \cite{gobet}. Let
\begin{align*}
 A:=[0 \leq \sigma \underset{0 \leq t \leq T}{\sup}(W_t^H - \theta t) \leq u].
\end{align*}
Since $\phi'(\sigma \underset{0\leq t \leq T}{\sup} (W_t^H -\theta t))=0$ on $A$ from the assumption of the function $\phi$, we get 
\begin{align*}
\left(D_s^{W_H} \phi \left(\sigma \underset{0 \leq t \leq T}{\sup}(W_t^H - \theta t)\right) \right) \psi(Y_t)&=\phi' \left( \sigma \underset{0\leq t \leq T}{\sup} (W_t^H -\theta t)\right) \sigma 1_{[s \leq \tau]} \psi(Y_t)\\
&=0\\
&=\phi'\left(\sigma \underset{0\leq t \leq T}{\sup} (W_t^H -\theta t)\right) \sigma \psi(Y_t),
\end{align*}
on $A$. On the other hand, on $A^c=[\sigma \underset{0\leq t \leq T}{\sup} (W_t^H -\theta t) > u]$, for any $\omega \in A^c \cap [\phi(Y_t) \neq 0]$ we have
\begin{align*}
\sigma \underset{0\leq t \leq T}{\sup} (W_t^H(\omega) -\theta t) > u, \quad Y_t(\omega) <\frac{u}{\sigma}.
\end{align*}
Therefore, we get
\begin{align*}
Y_t(\omega) < \frac{u}{\sigma} <  \underset{0\leq t \leq T}{\sup} (W_t^H(\omega) -\theta t))=W_{\tau}^H(\omega) - \theta \tau(\omega) \leq Y_{\tau} (\omega),
\end{align*}
so we have $t\leq \tau$ since $Y_t$ is an non-decreasing process. 
\end{proof}
In the sequel, we consider the smoothness of $u_A$ in the Malliavin sense. Let $\tilde{K}_H^*$ be the restriction of $K_H ^*$ to $L^2(0,T)$, and let $\tilde{K}_H^{*,adj}$ be the adjoint operator of $\tilde{K}_H^*$ in $L^2(0,T)$. Then, by Lanjri and Nualart \cite{Nualart3}, we have
 \begin{align*}
& \tilde{K}_H^{*,adj -1}(\psi(Y_{\cdot})(t)\\
&=d_Ht^{H-\frac{1}{2}}D_{0+}^{H-\frac{1}{2}}t^{\frac{1}{2}-H}\psi(Y_t)\\
&=\frac{d_H}{\Gamma(\frac{3}{2}-H)} \left( t^{\frac{1}{2}-H}\psi(Y_t)-\left(H-\frac{1}{2} \right) t^{H-\frac{1}{2}}\int_0^t \frac{t^{\frac{1}{2}-H}\psi(Y_t)-s^{\frac{1}{2}-H}\psi(Y_s)}{(t-s)^{H+\frac{1}{2}}}ds\right),
 \end{align*}
 and
 \begin{align*}
&\tilde{K}_H^{*-1}(u_{\cdot})(t)\\
&=\frac{1}{c_H \Gamma \left(H-\frac{1}{2} \right)} t^{\frac{1}{2}-H}D_{T-}^{H-\frac{1}{2}}\left( (\cdot)^{H-\frac{1}{2}}u(\cdot) \right)(t)\\
&=\frac{1}{c_H B\left(H-\frac{1}{2}, \frac{3}{2}-H\right)} \left\{ \frac{u_t}{(T-t)^{H-\frac{1}{2}}}+\left(H-\frac{1}{2}\right) t^{\frac{1}{2}-H} \int_t^T \frac{s^{H-\frac{1}{2}}u_s-t^{H-\frac{1}{2}}u_t}{(s-t)^{H+\frac{1}{2}}}ds\right\},
 \end{align*}
 where $d_H=\left(c_H\Gamma \left( H-\frac{1}{2}\right) \right)^{-1}$.
 So we get $u_A=\tilde{K}_H^{* -1} \circ \tilde{K}_H^{*,adj -1} (\psi(Y_{\cdot}) )$. Since $\tilde{K}_H^*:L^2(0,T)\rightarrow K_H^*[L^2(0,T)]$ is the  isometric isomorphism, the following lemma holds.
\begin{lemma}
\label{u}
For any stochastic process $u_t$ we have
\begin{align*}
u_{\cdot} \in\mathbb{D}_{K_H^*[L^2(0,T)]}^{1,p}(K_H^*[L^2(0,T)]) \Rightarrow \tilde{K}_H^{* -1}(u_{\cdot}) \in \mathbb{D}_{W^H}^{1,p}(\mathcal{H}).
\end{align*}
\end{lemma}
\begin{proof}
When $u_{\cdot}=Fv $ holds for $F \in \mathbb{D}_{K_H^*[L^2(0,T)]}^{1,p}$ and $ v \in K_H^*[L^2(0,T)]\subset L^2(0,T)$, we have
\begin{align*}
D^{W^H}(\tilde{K}_H^{* -1}(u_{\cdot}))&=D^{W^H}(F\tilde{K}_H^{* -1}(v_{\cdot}))\\
&=D^{W^H}F \otimes \tilde{K}_H^{* -1}(v)\\
&=\tilde{K}_H^{* -1}(D^{K_H^*[L^2(0,T)]}F) \otimes  \tilde{K}_H^{* -1}(v)\\
&=\tilde{K}_H^{* -1} \otimes \tilde{K}_H^{* -1}(D^{K_H^*[L^2(0,T)]}F\otimes v)\\ 
&=\tilde{K}_H^{* -1} \otimes \tilde{K}_H^{* -1}(D^{K_H^*[L^2(0,T)]}u_{\cdot}).
\end{align*}
When $u \in \mathbb{D}_{K_H^*[L^2(0,T)]}^{1,p}(K_H^*[L^2(0,T)])$, there exists $u_n \in S_{K_H^*[L^2(0,T)]}(K_H^*[L^2(0,T)])$ such that $u_n \rightarrow u \quad in \quad\mathbb{D}_{K_H^*[L^2(0,T)]}^{1,p}(K_H^*[L^2(0,T)])$, and we have
\begin{align*}
\tilde{K}_H^{* -1}u_n \rightarrow \tilde{K}_H^{* -1}u \quad in \quad L^p(\Omega; \mathcal{H}),
\end{align*}
since $\tilde{K}_H^{*} $ is the isometric isomorphism. Thus we have
\begin{align*}
&\left\|D^{W^H}(\tilde{K}_H^{* -1}(u_n))-\tilde{K}_H^{* -1} \otimes \tilde{K}_H^{* -1}(D^W(u)\right\|_{L^p(\Omega;\mathcal{H}\otimes \mathcal{H})}\\
&=\left\| \tilde{K}_H^{* -1} \otimes \tilde{K}_H^{* -1}(D^{K_H^*[L^2(0,T)]}u_n-D^{K_H^*[L^2(0,T)]}u)\right\|_{L^p(\Omega;\mathcal{H}\otimes \mathcal{H})}\\
&\leq \left\|D^{K_H^*[L^2(0,T)]}u_n- D^{K_H^*[L^2(0,T)]}u\right\|_{L^p(\Omega; L^2(0,T) \otimes L^2(0,T))} \\
% \tilde{K}_H^{* -1} \otimes \tilde{K}_H^{* -1} : K_H^{*}[L^2(0,T)] \otimes K_H^*[L^2(0,T)] \rightarrow L^2(0,T) \otimes L^2(0,T) : bounded)\\
&\rightarrow 0,
\end{align*}
and so we obtain $K_H^{* -1}(u_{\cdot}) \in \mathbb{D}_{W^H}^{1,p}(\mathcal{H})$.
\end{proof}

\begin{proposition}
 For $u_A$ given in \eqref{ua}, it holds that 
\begin{align*}
u_A \in \mathbb{D}_{W^H}^{1, \infty}(\mathcal{H}).
\end{align*}
\end{proposition}
\begin{proof}
It is sufficient to show that $\int_0^t \frac{t^{\frac{1}{2}-H}\psi(Y_t)-s^{\frac{1}{2}-H}\psi(Y_s)}{(t-s)^{H+\frac{1}{2}}}ds \in \mathbb{D}_{W^H}^{1, \infty}\left(\tilde{K}_H^*\left(L^2(0,T)\right)\right)$ from Lemma \ref{u} and $\psi(Y_t) \in  \mathbb{D}_{W^H}^{1,p}(\mathcal{H})$. Defining  $a_n$ and $b_n$ by
\begin{align*}
a_n+b_n
&:=\sum_{i=0}^{n-1} \frac{t^{\frac{1}{2}-H}\psi(Y_t) - \left(\frac{t}{2}+\frac{it}{2n} \right)^{\frac{1}{2}-H}\psi(Y_{\frac{t}{2}+\frac{it}{2n}})}{\left( t-\left(\frac{t}{2}+\frac{it}{2n}\right) \right)^{H+\frac{1}{2}}}\frac{t}{2n} \\
&\quad + \sum_{i=0}^{n-1} \frac{t^{\frac{1}{2}-H}\psi(Y_t) - \left(\frac{it}{2n}\right)^{\frac{1}{2}-H}\psi(Y_{\frac{it}{2n}})}{\left( t-\left(\frac{it}{2n}\right) \right)^{H+\frac{1}{2}}} \frac{t}{2n},
\end{align*}
then we can show
\begin{align*}
a_n &\rightarrow \int_{\frac{t}{2}}^t \frac{t^{\frac{1}{2}-H}\psi(Y_t)-s^{\frac{1}{2}-H}\psi(Y_s)}{(t-s)^{H+\frac{1}{2}}}ds \quad \text{a.s.},\\
b_n &\rightarrow \int_0^{\frac{t}{2}} \frac{t^{\frac{1}{2}-H}\psi(Y_t)-s^{\frac{1}{2}-H}\psi(Y_s)}{(t-s)^{H+\frac{1}{2}}}ds \quad \text{a.s.},
\end{align*}
as $n \rightarrow \infty$. We first show that 
\begin{align*}
a_n &\xrightarrow{L^p(\Omega)} \int_{\frac{t}{2}}^t \frac{t^{\frac{1}{2}-H}\psi(Y_t)-s^{\frac{1}{2}-H}\psi(Y_s)}{(t-s)^{H+\frac{1}{2}}}ds \quad (n\rightarrow \infty),
\end{align*}
holds and that $D^{W_H}a_n$ converges in $L^p(\Omega)$. Defining
\begin{align*}
A_T:=8\left(4 \int_0^T \int_0^T \frac{|W_s^H- \theta s -(W_u^H- \theta u)|^r}{|s-u|^{m+2}}dsdu \right)^{\frac{1}{r}} \frac{m+2}{m},
\end{align*}
we can evaluate as
\begin{align*}
|a_n|
&\leq \left| \sum_{i=0}^{n-1} \frac{t^{\frac{1}{2}-H}\psi(Y_t) - \left(\frac{t}{2}+\frac{it}{2n} \right)^{\frac{1}{2}-H}\psi(Y_t)}{\left( t-\left(\frac{t}{2}+\frac{it}{2n}\right) \right)^{H+\frac{1}{2}}}\frac{t}{2n} \right|\\
&+\left| \sum_{i=0}^{n-1} \frac{\left(\frac{t}{2}+\frac{it}{2n}\right)^{\frac{1}{2}-H}\psi(Y_t) - \left(\frac{t}{2}+\frac{it}{2n} \right)^{\frac{1}{2}-H}\psi(Y_{\frac{t}{2}+\frac{it}{2n}})}{\left( t-\left(\frac{t}{2}+\frac{it}{2n}\right) \right)^{H+\frac{1}{2}}}\frac{t}{2n}\right|\\
&\underset{\sim}{<} \left| \sum_{i=0}^{n-1} \frac{\left(\frac{t}{2} \right)^{-\frac{1}{2}-H}\psi(Y_t)}{\left( t-\left(\frac{t}{2}+\frac{it}{2n}\right) \right)^{H-\frac{1}{2}}}\frac{t}{2n} \right| + \left| \sum_{i=0}^{n-1} \frac{\left(\frac{t}{2} \right)^{\frac{m}{r}-\frac{1}{2}-H}A_T}{\left( t-\left(\frac{t}{2}+\frac{it}{2n}\right) \right)^{H-\frac{1}{2}}}\frac{t}{2n} \right|\\
&\underset{\sim}{<} \left| \sum_{i=0}^{n-1} \frac{\psi(Y_t)}{\left( t-\left(\frac{t}{2}+\frac{it}{2n}\right) \right)^{H-\frac{1}{2}}}\frac{t}{2n} \right| + \left| \sum_{i=0}^{n-1} \frac{A_T}{\left( t-\left(\frac{t}{2}+\frac{it}{2n}\right) \right)^{H-\frac{1}{2}}}\frac{t}{2n} \right|\\
&=: f_n^1(t)+f_n^2(t).
\end{align*}
Note that  $\left(f_n^1(t)\right)^p \xrightarrow{L^1(\Omega)} \left(\int_{\frac{t}{2}}^t \frac{\psi(Y_t)}{(t-s)^{H-\frac{1}{2}}}ds \right)^p$. Indeed, we have 
\begin{align*}
&\left\| f_n^1(t)- \int_{\frac{t}{2}}^t \frac{\psi(Y_t)}{(t-s)^{H-\frac{1}{2}}}ds \right\|_{L^p(\Omega)}\\
&\leq \|\psi(Y_t)\|_{L^p(\Omega)} \left| \sum_{i=0}^{n-1} \frac{1}{\left(t-\left(\frac{t}{2}+\frac{it}{2n}\right)\right)^{H-\frac{1}{2}}}\frac{t}{2n} -\int_{\frac{t}{2}}^t \frac{1}{(t-s)^{H-\frac{1}{2}}}ds \right|\\
&\rightarrow 0 \quad(n \rightarrow \infty).
\end{align*}
 Similarly, $\left(f_n^2\right)^p \xrightarrow{L^1(\Omega)} \left(\int_{\frac{t}{2}}^t \frac{A_T}{(t-s)^{H-\frac{1}{2}}}ds \right)^p$ holds, so we obtain
 \begin{align}
 \label{an}
 a_n \xrightarrow{L^p(\Omega)}  \int_{\frac{t}{2}}^t \frac{t^{\frac{1}{2}-H}\psi(Y_t) - s^{\frac{1}{2}-H}\psi(Y_s)}{\left( t-s \right)^{H+\frac{1}{2}}}ds \quad(n\rightarrow \infty).
 \end{align}
Also, since
\begin{align*}
D^{W^H}a_n=\sum_{i=0}^{n-1} \frac{D^{W^H}A_T\left( t^{\frac{1}{2}+\frac{m}{r}-H}\psi'(Y_t) - \left(\frac{t}{2}+\frac{it}{2n} \right)^{\frac{1}{2}+\frac{m}{r}-H}\psi'(Y_{\frac{t}{2}+\frac{it}{2n}})\right)}{\left( t-\left(\frac{t}{2}+\frac{it}{2n}\right) \right)^{H+\frac{1}{2}}}\frac{t}{2n},
\end{align*}
we have
\begin{align*}
\left| D^{W^H}a_n\right| &\underset{\sim}{<}  \left| \sum_{i=0}^{n-1} \frac{D^{W^H}A_T\left( \frac{t}{2} \right)^{\frac{m}{r}-H-\frac{1}{2}} \psi '(Y_t)}{\left(t-\left(\frac{t}{2}+\frac{it}{2n}\right) \right)^{H-\frac{1}{2}}} \right|+\left|\sum_{i=0}^{n-1} \frac{D^{W^H}A_T\left(\frac{t}{2}\right)^{2\frac{m}{r}-H-\frac{1}{2}}A_T}{(t-s)^{H-\frac{1}{2}}}\right|\\
&\underset{\sim}{<} \left| \sum_{i=0}^{n-1} \frac{D^{W^H}A_T\psi '(Y_t)}{\left(t-\left(\frac{t}{2}+\frac{it}{2n}\right) \right)^{H-\frac{1}{2}}} \right|+\left|\sum_{i=0}^{n-1} \frac{A_TD^{W^H}A_T}{(t-s)^{H-\frac{1}{2}}}\right|.
\end{align*}
Thus we get
\begin{align*}
D^{W^H}a_n \xrightarrow{L^p(\Omega)} \int_{\frac{t}{2}}^t \frac{D^{W^H}A_T \left(t^{\frac{m}{r}+\frac{1}{2}-H}\psi'(Y_t) - s^{\frac{m}{r}+\frac{1}{2}-H}\psi'(Y_s)\right)}{\left( t-s \right)^{H+\frac{1}{2}}}ds \quad(n\rightarrow \infty)
\end{align*}
in the same way as in $(\ref{an})$. Second, we show that 
\begin{align*}
b_n &\xrightarrow{L^p(\Omega)} \int_0^{\frac{t}{2}} \frac{t^{\frac{1}{2}-H}\psi(Y_t)-s^{\frac{1}{2}-H}\psi(Y_s)}{(t-s)^{H+\frac{1}{2}}}ds \quad (n\rightarrow \infty),
\end{align*}
holds and that $D^{W_H}b_n$ converges in $L^p(\Omega)$. Since it holds that
\begin{align}
\left|b_n \right| &\leq \left(\frac{2}{t} \right)^{H+\frac{1}{2}} \left\{ \sum _{i=0}^{n-1} \left|t^{\frac{1}{2}-H} \psi(Y_t) -\left(\frac{it}{2n}\right)^{\frac{1}{2}-H}\psi(Y_{\frac{it}{2n}})\right| \right\} \frac{t}{2n} \nonumber \\
&\underset{\sim}{<} \left(\frac{2}{t} \right)^{H+\frac{1}{2}} \left\{ t^{\frac{1}{2}-H}\psi(Y_t) + \sum_{i=0}^n \left( \frac{it}{2n} \right)^{\frac{m}{r}+\frac{1}{2}-H}A_T\frac{t}{2n} \right\}\nonumber \\
\label{bn}
&=:\left(\frac{2}{t} \right)^{H+\frac{1}{2}} f_n^3(t),
\end{align}
we have
\begin{align*}
&\left\| f_n^3(t) - \int_0^{\frac{t}{2}} \left(t^{\frac{1}{2}-H}\psi(Y_t) - s^{\frac{1}{2}-H}A_T \right) ds \right\|_{L^p(\Omega)} \\
&\underset{\sim}{<} \|A_T\|_{L^p(\Omega)} \left| \sum_{i=0}^{n-1} \left(\frac{it}{2n} \right)^{\frac{m}{r}+\frac{1}{2}-H}\frac{t}{2n}-\int_0^{\frac{t}{2}} s^{\frac{m}{r}+\frac{1}{2}-H}ds\right|\\
&\rightarrow 0 \quad(n \rightarrow \infty).
\end{align*}
Thus we have $b_n \xrightarrow{L^p(\Omega)} \int_0^{\frac{t}{2}} \frac{t^{\frac{1}{2}-H} \psi(Y_t) -s^{\frac{1}{2}-H}\psi(Y_s)}{(t-s)^{H+\frac{1}{2}}}ds$. From $(\ref{bn})$, we obtain 
\begin{align*}
D^{W^H}b_n \xrightarrow{L^p(\Omega)} \int_0^{\frac{t}{2}} \frac{D^{W^H}A_T \left(t^{\frac{m}{r}+\frac{1}{2}-H}\psi'(Y_t) - s^{\frac{m}{r}+\frac{1}{2}-H}\psi'(Y_s)\right)}{\left( t-s \right)^{H+\frac{1}{2}}}ds \quad(n\rightarrow \infty)
\end{align*}
in the same way as for $a_n$.  Therefore, since $D^{W_H}$ is a closable operator, we have \[
\int_0^t \frac{t^{\frac{1}{2}-H}\psi(Y_t)-s^{\frac{1}{2}-H}\psi(Y_s)}{(t-s)^{H+\frac{1}{2}}}ds \in \mathbb{D}_{W^H}^{1, \infty}.
\]
\end{proof}

\begin{proof}[Proof of Theorem 2]
From Proposition 1 we have
\begin{align*}
&\phi'(\sigma \underset{0\leq t \leq T}{\sup} (W_t^H -\theta t)) \sigma\left\langle 1_{[0,\tau]},u_A\right\rangle_{\mathcal{H}}\\
&=\phi'(\sigma \underset{0\leq t \leq T}{\sup} (W_t^H -\theta t)) \sigma\left\langle \tilde{K}_H^*(1_{[0,\tau]}),  \tilde{K}_H^{*,adj -1} (\psi(Y_{\cdot}) ) \right\rangle_{L^2(0,T)}\\
&=\phi'(\sigma \underset{0\leq t \leq T}{\sup} (W_t^H -\theta t)) \sigma\left\langle 1_{[0,\tau]}, \psi(Y_{\cdot}) \right\rangle_{L^2(0,T)}\\
&=\phi'(\sigma \underset{0\leq t \leq T}{\sup} (W_t^H -\theta t)) \sigma\int_0^T \psi(Y_t)dt.
\end{align*}
Thus, since it holds that
\begin{align*}
&\frac{\partial}{\partial \sigma} \mathbb{E}\left[\phi\left(\sigma \sup_{0 \leq t \leq T}(W_t^H- \theta t) \right) \right]\\
&=\mathbb{E}\left[ \phi'(\sigma \underset{0\leq t \leq T}{\sup}(W_t^H- \theta t))\underset{0\leq t \leq T}{\sup}(W_t^H- \theta t)\right]\\
&=\mathbb{E}\left[ \phi'(\sigma \underset{0\leq t \leq T}{\sup}(W_t^H- \theta t))\underset{0\leq t \leq T}{\sup}(W_t^H- \theta t) \frac{\int_0^T \psi(Y_t)dt}{\int_0^T \psi(Y_t)dt} \right]\\
&=\mathbb{E} \left[ \phi'(\sigma \underset{0\leq t \leq T}{\sup}(W_t^H- \theta t)) \sigma \underset{0\leq t \leq T}{\sup}(W_t^H- \theta t) \frac{\left\langle 1_{[0,\tau]}, u_A(\cdot) \right\rangle_{\mathcal{H}}}{\sigma\int_0^T \psi(Y_t)dt} \right]\\
&=\mathbb{E} \left[ \left\langle D^{W^H}(\phi(\sigma \sup_{0\leq t \leq T} (W_t^H- \theta t))),   \frac{u_{A}(\cdot) \underset{0 \leq t \leq T}{\sup}(W_t^H- \theta t)}{\sigma \int_0^T \psi(Y_t)dt} \right\rangle_{\mathcal{H}} \right],
\end{align*}
and we have $ \frac{u_{A}(\cdot) \underset{0 \leq t \leq T}{\sup}(W_t^H- \theta t)}{\sigma \int_0^T \psi(Y_t)dt} \in \text{Dom} \delta^{W^H}$, we get $(\ref{submain})$. Next we show that $(\ref{main result})$ holds. Defining $g_n: \mathbb{R}_{+} \rightarrow \mathbb{R}$ by
\begin{align*}
g_n:=1_{[u+\frac{1}{n}, u+n+\frac{1}{n}]}*\rho_{\frac{1}{n}} \quad(n\geq2,\  \rho : \text{molifier}),
\end{align*}
then we have
\begin{itemize}
\item[1] $g_n \in C_b^{\infty}(\mathbb{R})$,
\item[2] $g_n(x)=0 \quad \text{on} \  [0,u]$.
\end{itemize}
Also we get $g_n \rightarrow 1_{(u, \infty)}$ as $ n \rightarrow \infty$ and
\begin{align*}
f_n(\sigma)&:=\mathbb{E}[g_n(\sigma \underset{0 \leq t \leq T}{\sup}(W_t^H- \theta t))]\\
&\rightarrow \mathbb{E}[1_{(u, \infty)}(\sigma \underset{0 \leq t \leq T}{\sup}(W_t^H- \theta t))] \\
&=\mathbb{E}[1_{[u, \infty)}(\sigma \underset{0 \leq t \leq T}{\sup}(W_t^H- \theta t))] .
\end{align*}
Therefore, for any compact set $K \subset \mathbb{R}_{+}$, we have
\begin{align}
&\left| \underset{\sigma \in K}{\sup} \left( \frac{\partial}{\partial \sigma}f_n(\sigma) - \mathbb{E}\left[ 1_{[u, \infty)}(\sigma \underset{0 \leq t \leq T}{\sup}(W_t^H- \theta t))\delta^{W^H}\left( \frac{u_A(\cdot) \underset{0 \leq t \leq T}{\sup}(W_t^H- \theta t)}{\int_0^T \sigma  \psi(Y_t) dt}  \right)\right] \right) \right| \nonumber\\
&\leq \frac{1}{\underset{\sigma \in K}{\inf} \sigma} \left\| \delta^{W^H} \left( \frac{u_A(\cdot) \underset{0 \leq t \leq T}{\sup}(W_t^H- \theta t)}{\int_0^T  \sigma \psi(Y_t) dt} \right)\right\|_{L^2(\Omega)} \underset{\sigma \in K}{\sup} \mathbb{E}\left[ \left( g_n(\sigma V_T^*) - 1_{[u, \infty)} \right)^2\right]^{\frac{1}{2}}  \nonumber\\
&\underset{\sim}{<}  \underset{\sigma \in K}{\sup} \left\{ \mathbb{P} \left( \sigma V_T^* \in [u, u+\frac{2}{n}] \right) + \mathbb{P} \left( \sigma V_T^* \in [u+n, \infty ) \right) \right\}\nonumber \\
\label{main result2}
&=:\underset{\sigma \in K}{\sup}\left\{ h_n^1(\sigma) +h_n^2(\sigma)\right\}.
\end{align}
In considering the compact uniform convergence of $h_n^1(\sigma)$ and $h_n^2(\sigma)$ with respect to $\sigma$,  it is sufficient to show only the continuity of $h_n^1(\sigma)$ and $h_n^2(\sigma)$ with respect to $\sigma$ by the Dini theorem. The problem here is that when $\sigma$ changes, the interval $\left[ \frac{u}{\sigma} , \frac{u+\frac{2}{n}}{\sigma} \right]$  also moves, so the continuity of the measure $\mathbb{P}$ cannot be exploited. Therefore, for any $(\sigma_m) \subset K$ such that $\sigma_m \downarrow \sigma \quad (m \rightarrow \infty)$, we shall show that for any $m\in \mathbb{N}$ there exists a fixed point that in $\left[ \frac{u}{\sigma_m} , \frac{u+\frac{2}{n}}{\sigma_m} \right]$.
For $(\sigma_m) \in K$ such that $\sigma_m \downarrow \sigma\quad(m\rightarrow \infty)$, take $\varepsilon > 0$ satisfying $\varepsilon < \frac{\frac{2}{n \sigma}}{2u-\frac{2}{n}}$ and $N \in \mathbb{N}$ large enough to satisfy $\left| \frac{1}{\sigma_N} -\frac{1}{\sigma} \right| < \varepsilon$. Then, it holds that
\begin{align*}
\frac{u}{\sigma_N} < u\left(\frac{1}{\sigma} + \varepsilon\right)< \left( u + \frac{2}{n} \right) \left(\frac{1}{\sigma}- \varepsilon \right)<\frac{u+\frac{2}{n}}{\sigma_N}.
\end{align*}
from $\left(2u+\frac{2}{n} \right) \varepsilon < \frac{2}{n \sigma}$ and $\left| \frac{1}{\sigma_N} -\frac{1}{\sigma} \right| < \varepsilon$. Therefore, we can show that $h_n^1(\sigma)$ is right-continuous with respect to $\sigma$ because we can take $a:=u\left(\frac{1}{\sigma}+\varepsilon\right)$ independent of $N \in \mathbb{N}$ such that $\frac{u}{\sigma_N} <a <\frac{u+\frac{2}{n}}{\sigma_N}$. In the same way, we can show the case of $\sigma_m \uparrow \sigma\quad(m \rightarrow \infty)$, so we obetain the continuity of $h_n^1(\sigma)$ and $h_n^2(\sigma)$. Thus, $(\ref{main result2})$  converges to 0 as $n\rightarrow \infty$, and this proof is complete.
\end{proof}

\ \vspace{3mm}\\
\begin{flushleft}
{\bf\large Acknowledgements.}
The authors express sincere thanks to Prof. A. Kohats-Higa for the valuable discussion related to the part of Malliavin calculus. 
%The author also thanks anonymous referees for detailed suggestions and proposals that make the paper improve extensively.  
This research was partially supported by JSPS KAKENHI Grant-in-Aid for Scientific Research (C)  \#21K03358. 
\end{flushleft}


\begin{thebibliography}{99}
\bibitem{aa10} Asmussen, S and Albrecher, H. (2010). {\it Ruin probabilities}. 2nd ed. World Scientific Publishing Co. Pte. Ltd., Hackensack, NJ. 
\bibitem{Billingsley} Billingsley, P. (1999). \textit{Convergence of probability measures}. 2nd ed. John Wiley \& Sons, New
York.
\bibitem{Cai} Cai, C. and Xiao, W. (2021). Simulation of an integro-differential equation and application in estimation of ruin probability with mixed fractional Brownian motion. J. Integral Equations Applications, \textbf{33}, (1), 1-17.
\bibitem{Nualart1} Corcuera, J. M. Nualart, D. (2006). Power variation of some integral fractional processes. \textit{Bernoulli} \textbf{12}, no. 4, 713-735. 
\bibitem{Florit} Florit, C. and Nualart, D. (1995). A local criterion for smoothness of densities and application to
the supremum of the Brownian sheet. \textit{Statistics and Probability Letters} \textbf{22}, 25-31.
\bibitem{Gerber} Gerber, H.U., Shiu, E.S.W. (1998). On the time value of ruin. \textit{N. Am. Actuar.} \textbf{2}, 48-72. 
\bibitem{gobet} Gobet, E. and Kohatsu-Higa, A. (2003). Computation of Greeks for barrier and look-back options using Malliavin calculus. \textit{Electron. Comm. Probab}. \textbf{8}, 51-62. 
\bibitem{Ji} Ji, L. and Robert, R. (2018). Ruin problem of a two-dimensional fractional Brownian motion risk process. \textit{Stoch. Models}, \textbf{34}, (1), 73-97.
\bibitem{Nualart3} Lanjri Z. N. and Nualart, D. (2003). Smoothness of the law of the supremum of the fractional Brownian motion.  \textit{Electron. Comm. Probab.} \textbf{8}, 102-111.
\bibitem{Lundberg} Lundberg, F. (1903). Approximerad framst\"{a}llning av sannolikhetsftmktionen. Aterfors\"{a}knng av kollektivrisker. \textit{Akad Afhandling. Almqvist och Wiksell, Uppsala.}
\bibitem{Michna} Michna, Z. (1998). Self-similar processes in collective risk theory, \textit{J. Appl. Math. Stochastic
Anal.} \textbf{11}, (4), 429-448.
\bibitem{Norros} Norros, I.; Valkeila, E. and Virtamo, J. (1999). An elementary approach to a Girsanov formula and other analytical results on fractional Brownian motion. \textit{Bernoulli}, \textbf{5}, (4), 571-587.
\bibitem{Nualart2} Nualart, D. (1995). \textit{Malliavin calculus and related topics}. (Probability and its Applications). Berlin Heidelberg New York Springer.
\bibitem{Samko} Samko, S. G.; Kilbas, A. A. and Marichev, O. I. (1994).  \textit{Fractional integrals and derivatives. Theory and applications.} Gordon and Breach Science, Yverdon.
\bibitem{Shimizu1} Shimizu, Y. (2021), \textit{Asymptotic statistics in insurance risk theory.} SpringerBriefs in Statistics.
\bibitem{van der varrt} van der Vaart, A. W. and Wellner, J. A. (1996). \textit{Weak convergence and empirical processes}. With applications to statistics. Springer Series in Statistics. Springer-Verlag, New York.

\end{thebibliography}
\end{document}